\documentclass[11pt, a4paper]{article}

\usepackage[whole]{bxcjkjatype}

\usepackage{bm, amsmath, amsthm, amssymb, accents, comment}
\usepackage{ascmac}
\RequirePackage{amsthm,amsmath,amsfonts,amssymb}
\RequirePackage[numbers]{natbib}
\RequirePackage{graphicx}\usepackage{tikz}
\usepackage{pgfplots}
\usepackage{subcaption}
\pgfplotsset{
	compat=newest, 
	cycle list name=exotic }

\DeclareMathOperator*{\argmin}{arg\,min}

\usepackage{graphicx}
\usepackage{xcolor}\usepackage{multirow}\usepackage{amsmath,amssymb,amsfonts}\usepackage{amsthm}\usepackage{mathrsfs}\usepackage[title]{appendix}\usepackage{textcomp}\usepackage{manyfoot}\usepackage{booktabs}\usepackage{algorithm}\usepackage{algorithmicx}\usepackage{algpseudocode}\usepackage{listings}\usepackage{url}
\usepackage{mathrsfs}
\usepackage{physics}
\usepackage{bm}
\usepackage{float}
\usepackage{hyperref}
\usepackage[numbers]{natbib}
\usepackage{tikz}
\usepackage{pgfplots}
\usepackage{pgfplotstable}
\usetikzlibrary{arrows,positioning,plotmarks,external,patterns,angles,decorations.pathmorphing,backgrounds,fit,shapes,graphs,calc,spy}
\pgfplotsset{compat=1.18}

\newtheorem{theorem}{Theorem}

\usepackage{typearea}
\typearea{12}

\pgfplotstableread[
col sep=comma,
]{./csv_data/vonmises_MSE/ex6_mu0.3kappa2.csv}\vMn

\pgfplotstableread[
col sep=comma,
]{./csv_data/vonmises_MSE/ex65_MSE_kappa_change.csv}\vMk

\pgfplotstableread[
col sep=comma,
]{./csv_data/wrapcauchy_MSE/ex7_wrapcauchy_MSE.csv}\wCn

\pgfplotstableread[
col sep=comma,
]{./csv_data/wrapcauchy_MSE/ex75_wrapcauchy_change_rho.csv}\wCrho

\pgfplotstableread[
col sep=comma,
]{./csv_data/ss_vonmises_MSE/ex9_ss_vonmises_MSE.csv}\svMn

\pgfplotstableread[
col sep=comma,
]{./csv_data/ss_vonmises_MSE/ex95changelambda.csv}\svMlambda

\pgfplotstableread[
col sep=comma,
]{./csv_data/vonmises_mix/ex8_vonmises_mix_MSE.csv}\robust

\pgfplotstableread[
col sep=comma,
]{./csv_data/vonmises_mix/ex85_change_noise_rate.csv}\robustt

\pgfplotstableread[
col sep=comma,
]{./csv_data/model_misspecification/ex10_vM_WC.csv}\vMWC

\pgfplotstableread[
col sep=comma,
]{./csv_data/model_misspecification/ex10_WC_vM.csv}\WCvM

\begin{document}
\title{Wasserstein projection estimators for circular distributions}
	
	\author{Naoki Otani\thanks{Department of Mathematical Informatics, The University of Tokyo}, \ 
		Takeru Matsuda\thanks{Department of Mathematical Informatics, The University of Tokyo \& Statistical Mathematics Unit, RIKEN Center for Brain Science, e-mail: \texttt{matsuda@mist.i.u-tokyo.ac.jp}}}
	
	\date{}
	
	\maketitle
	
	\begin{abstract}
	For statistical models on circles, we investigate performance of estimators defined as the projections of the empirical distribution with respect to the Wasserstein distance.
We develop algorithms for computing the Wasserstein projection estimators based on a formula of the Wasserstein distances on circles.
Numerical results on the von Mises, wrapped Cauchy, and sine-skewed von Mises distributions show that the Wasserstein projection estimators attain estimation accuracy comparable to the maximum likelihood estimator.
In addition, the Wasserstein projection estimators are found to be robust against noise contamination.
	\end{abstract}

	\section{Introduction}
Let $S$ be a metric space with distance $d$.
For $p \geq 1$, the $L^p$ Wasserstein distance between two probability distributions $p_1$ and $p_2$ on $S$ is defined by
\[
W_p (p_1,p_2) = \inf_{X,Y} \ {\rm E} [ d(X,Y)^p ] ^{1/p},
\]
where the infimum is taken over all joint distributions (coupling) of $(X,Y)$ with marginal distributions of $X$ and $Y$ equal to $p_1$ and $p_2$, respectively.
The Wasserstein distance can be interpreted as the optimal transportation cost from $p_1$ to $p_2$ \cite{Villani2009}.
Unlike other dispersion measures between probability distributions such as the Kullback--Leibler divergence and the $L^p$ distance, the Wasserstein distance has the characteristic that it inherits the metric structure of the underlying space.
By exploiting this property, the Wasserstein distance has been widely used in machine learning and computer vision \citep{PC2019}.

Several studies have investigated the application of the Wasserstein distance to parameter estimation of statistical models \citep{BBR2006,BJGR2019,FZMAP2015,MMC2015}.
Specifically, suppose that we have $n$ independent samples $X_1,\dots,X_n$ from $p(x \mid \theta)$ and estimate $\theta$ based on them.
Then, the $L^p$ Wasserstein projection estimator is defined as 
\[
\hat{\theta}_{W,p} = \argmin_{\theta} W_p (\hat{q}, p_{\theta}),
\]
where $\hat{q}$ is the empirical distribution of $X_1,\dots,X_n$.
It can be interpreted as the projection of the empirical distribution onto the model with respect to the Wasserstein distance.
Note that the maximum likelihood estimator (MLE) can be interpreted as the projection of the empirical distribution with respect to the Kullback--Leibler divergence.

Recently, \cite{AM2022} studied the $L^2$ Wasserstein projection estimator on the real line ($S=\mathbb{R}$).
For location-scale families, they showed that the Wasserstein projection estimator is given by a linear combination of the order statistics, which is often called the L-statistic \cite{vV}, and derived its asymptotic distribution.
Notably, the $L^2$ Wasserstein projection estimator attains statistical efficiency in Gaussian case.
These results are based on the fact that the optimal transport on the real line is explicitly obtained by using the cumulative distribution function.
Thus, it cannot be extended to other spaces straightforwardly.

In this study, we investigate the Wasserstein projection estimator for statistical models on circles ($S=\mathbb{S}^1$).
Circular data appear in many fields such as meteorology (e.g. wind direction), ecology (e.g. phenological pattern), and neuroscience (e.g. neural oscillation phase).
Methods for analyzing such circular data have been widely studied in directional statistics \cite{Mardia} and many parametric models on circles have been proposed such as von Mises, wrapped Cauchy and sine-skewed von Mises distributions \cite{skewed}.
Based on a formula for the Wasserstein distance on circles \cite{Delon,Rabin}, we develop algorithms for computing the Wasserstein projection estimator on circles.
Then, we conduct numerical experiments on von Mises, wrapped Cauchy, and sine-skewed von Mises distributions.
The Wasserstein projection estimators attain estimation accuracy comparable to the maximum likelihood estimator.
In addition, the Wasserstein projection estimators are found to be robust against noise contamination.
Simulation under model mis-specification showed different properties of the Wasserstein projection estimators and the maximum likelihood estimator.

\section{Wasserstein projection estimators on circles}
In this section, we introduce the Wasserstein projection estimators on circles and develop algorithms for computing them.
We identify the circle $\mathbb{S}^1=\{ (\cos \theta, \sin \theta) \mid 0 \leq \theta < 2\pi \}$ with $[0, 2\pi)$ in the following.
It is endowed with the metric $d(x,y) = \min (|x-y|,2\pi-|x-y|)$ for $x,y \in [0, 2\pi)$.

\subsection{Wasserstein distance on circles}
We first review the formula for the Wasserstein distance on circles by \cite{Delon,Rabin}.
Following \cite{Rabin}, we define the cumulative distribution function $Q$ of a probability distribution $q$ on the circle by 
\[
Q(x) = \int_0^x q(z) {\rm d} z
\]
for $0 \leq x < 2 \pi$ and $Q(x \pm 2\pi) = Q(x) \pm 1$ otherwise.
Note that $Q$ is a non-decreasing function on $\mathbb{R}$.
Then, the $L^p$ Wasserstein distance on the circle is given by
\[
W_p (p_1,p_2) = \left( \min_{\alpha \in {R}} \int_0^1 |P_1^{-1} (u) - P_{2,\alpha}^{-1} (u)|^p {\rm d} u \right)^{1/p},
\]
where $P_{2,\alpha}(x)=P_2(x)-\alpha$ \cite{Rabin}.
Since the objective function in the right-hand side is convex with respect to $\alpha$, computation of $W_p (p_1,p_2)$ is reduced to one-dimensional convex optimization, which can be solved by binary search.
For $p=1$, the above formula can be rewritten as
\[
W_1 (p_1,p_2) = \min_{\alpha \in \mathbb{R}} \int_0^{2 \pi} |P_1 (x) - P_2 (x) - \alpha| {\rm d} x.
\]

\subsection{Algorithm}
Let $p(x \mid \theta)$ be a parametric model on the circle (e.g., von Mises, wrapped Cauchy).
Suppose that we have $n$ independent samples $X_1,\dots,X_n$ from $p(x \mid \theta)$ and estimate $\theta$ based on them.
The $L^p$ Wasserstein projection estimator is defined as 
\[
\hat{\theta}_{W,p} = \argmin_{\theta} W_p (\hat{q}, p_{\theta}),
\]
where $\hat{q}$ is the empirical distribution of $X_1,\dots,X_n$.

Unlike the location-scale families on the real line \citep{AM2022}, it seems difficult to obtain $\hat{\theta}_{W,p}$ in closed form. 
Thus, we compute $\hat{\theta}_{W,p}$ by minimizing $W_p (\hat{q}, p_{\theta})$ with respect to $\theta$ by using Powell's method or differential evolution.
Below we develop two algorithms for computing $W_p (\hat{q}, p_{\theta})$.
Whereas the first one is applicable to general $p \geq 1$, the second one is specialized to $p=1$ but computationally more effiicient.
Note that parallel numerical solvers would make the computation more efficient \cite{Li}.

In the first algorithm ($p \geq 1$), we approximate $p(x \mid \theta)$ by a $n$-point discrete distribution $\hat{p}(x \mid \theta)$ with equal weights $1/n$ on $P^{-1}(k/n \mid \theta)$ for $k=1,\dots,n$.
Since $W_p (\hat{q}, \hat{p}_{\theta})$ is the Wasserstein distance between two discrete distributions, it can be computed in $O(n \log (\varepsilon^{-1}))$ time for sorted inputs \cite{Delon}.
From the triangle inequality, the discretization error is bounded as $|W_p(\hat{q},\hat{p}_{\theta})-W_p(\hat{q},{p}_{\theta})| \leq W_p(\hat{p}_{\theta},{p}_{\theta})$, which converges to zero as $n \to \infty$.

In the second algorithm ($p=1$), we approximate $p(x \mid \theta)$ by a $D$-point discrete distribution $\tilde{p}(x \mid \theta)$ with weights $P^{-1}(k/D)-P^{-1}((k-1)/D)$ on $2\pi \cdot k/D$ for $k=1,\dots,D$.
We also approximate $\hat{q}(x)$ by $\tilde{q}(x)$ similarly.
Note that $W_1(\tilde{q},\tilde{p}_{\theta})$ is efficiently computed by using the formula
\[
W_1 (\tilde{q},\tilde{p}_{\theta} ) = \frac{2 \pi}{D} \sum_{i=1}^{D} \left| \tilde{Q} \left( 2 \pi \frac{i}{D} \right) - \tilde{P} \left( 2 \pi \frac{i}{D} \mid \theta \right) - m \right|,
\]
where $m$ is the median of $\tilde{Q} ( k/D ) - \tilde{P} ( k/D \mid \theta )$ for $k=1,\dots,D$ and obtained by the median of medians algorithm in $O(D)$ time \cite{Hundrieser}. 
From $W_1(\tilde{p}_{\theta},p_{\theta}) \leq 2\pi/D$ and $W_1(\tilde{q},\hat{q}) \leq 2\pi/D$, we have $|W_1(\tilde{q},\tilde{p}_{\theta})-W_1(\hat{q},p_{\theta})| \leq 4 \pi/D$.
We set the grid size to $D = n$.
This choice reflects a trade-off between computational cost and approximation accuracy:
the computational complexity scales as $O(n + n_{\mathrm{feval}} D)$ while the approximation error is bounded above by $O(1/D)$, where $n_{\mathrm{feval}}$ is the number of parameter search iterations.
Under this choice of the grid size, the discretization error vanishes as $n \to \infty$.

Table~\ref{tab} compares the computation time of MLE, $\hat{\theta}_{W,1}$ by the second algorithm, and $\hat{\theta}_{W,2}$ by the first algorithm for $n=10000$.

\begin{table}[htbp]
	\caption{Computation time (in seconds) of estimators for $n=10000$.}
	\label{tab}
	\centering
	\begin{tabular}{|c|c|c|}
		\hline
		& von Mises & wrapped Cauchy  \\ \hline
		MLE & $6.2 \times 10^{-4}$ & $2.6 \times 10^{-3}$  \\ \hline
		W1 & 7.5 & 0.15  \\ \hline
		W2 & 5.6 & 1.8  \\ \hline
	\end{tabular}
\end{table}

\subsection{Consistency}
From the general theory of M-estimation \cite{vV}, the Wasserstein projection estimator is shown to be consistent as follows.

\begin{theorem}
	If $p \geq 1$ and the model is parametric and identifiable, then $\hat{\theta}_{W,p} \to \theta$ as $n \to \infty$.
\end{theorem}
\begin{proof}
	From the triangle inequality,
	\[
	|W_p(\hat{q},p_{\theta}) - W_p(p_{\theta_0},p_{\theta})| \leq W_p(\hat{q},p_{\theta_0})
	\]
	for every $\theta$.
	From \cite{Fournier},
	\[
	W_p(\hat{q},p_{\theta_0}) \to 0
	\]
	as $n \to \infty$.
	Thus,
	\[
	\sup_{\theta} |W_p(\hat{q},p_{\theta}) - W_p(p_{\theta_0},p_{\theta})| \to 0
	\]
	as $n \to \infty$.
	Also, since the model is parametric and  identifiable, for every $\varepsilon>0$,
	\[
	\inf_{\theta: \| \theta-\theta_0 \| > \varepsilon} W_p(p_{\theta_0},p_{\theta}) > 0 = W_p(p_{\theta_0},p_{\theta_0}),
	\]
	Therefore, from Theorem~5.7 of \cite{vV}, $\hat{\theta} \to \theta_0$.
\end{proof}

We examine the performance of the Wasserstein projection estimator numerically in the next section.
It is an interesting future problem to derive the asymptotic distribution theoretically.

\section{Numerical results}
In this section, we present results of numerical experiments.
The python code is available at \url{https://github.com/naoppy/WassersteinProjEstimatorOnS1}.

\subsection{von Mises distribution}
Here, we consider the von Mises distribution defined by
\[
p(x \mid \mu,\kappa) = \frac{1}{2 \pi I_0(\kappa)}\exp (\kappa \cos (\theta-\mu)),
\]
where $\mu \in [0,2\pi)$ is the mean parameter and $\kappa>0$ is the concentration parameter.
Its Fisher information matrix is
\[
I(\mu,\kappa) = \begin{pmatrix} \kappa \frac{I_1(\kappa)}{I_0(\kappa)} & 0 \\ 0 & \frac{1}{2} + \frac{I_2(\kappa)}{2 I_0(\kappa)} - \left( \frac{I_1(\kappa)}{I_0(\kappa)} \right)^2 \end{pmatrix}.
\]
The maximum likelihood estimator is obtained in closed form by using the Bessel function \cite{Mardia}.

Figures~\ref{fig_vM} and \ref{fig_vM2} display the mean squared error (MSE) ratios of the $L^1$ and $L^2$ Wasserstein projection estimators to the maximum likelihood estimator, given by
\[
\frac{{\rm E}[(\hat{\mu}_{W,p}-\mu)^2]}{{\rm E}[(\hat{\mu}_{\mathrm{ML}}-\mu)^2]} \quad \text{and} \quad \frac{{\rm E}[(\hat{\kappa}_{W,p}-\kappa)^2]}{{\rm E}[(\hat{\kappa}_{\mathrm{ML}}-\kappa)^2]}.
\]
Specifically, the ratios are plotted against the sample size $n$ in Figure~\ref{fig_vM}, and against the concentration parameter $\kappa$ in Figure~\ref{fig_vM2}.
Note that the mean squared errors do not depend on $\mu$ due to shift invariance.
The Wasserstein projection estimators attain estimation accuracy comparable to the maximum likelihood estimator.

\citet{AM2022} showed that the $L^2$ Wasserstien projection estimator is efficient for one-dimensional Gaussian distributions.
Figures~\ref{fig_vM} and \ref{fig_vM2} indicate a similar property for the von Mises distribution, which can be viewed as a circular analogue of the one-dimensional Gaussian distribution for large $\kappa$.

\begin{figure}[htbp]
	\centering
	\begin{tikzpicture}
		\tikzset{every node/.style={}}
		\begin{axis}
			[width=6cm,
			xmax=5,xmin=2,
			ymax=1.1,ymin=0.98,
			xlabel={$\log_{10} n$},
			ylabel={MSE ratio},
			yticklabel style={
				/pgf/number format/.cd,
				fixed,
				zerofill,
				precision=3,
			},
			ylabel near ticks,
			legend pos=north east,
			legend style={legend cell align=left,draw=none,fill=white,fill opacity=0.8,text opacity=1,},
			]
			\addplot[color=black, mark=square*, mark options={solid}] table [x=log10N, y expr={pow(10, \thisrow{W1(method2)_mu} - \thisrow{MLE_mu})}] {\vMn};
			\addlegendentry{W1/MLE}
			\addplot[color=black, dashed, mark=*, mark options={solid}] table [x=log10N, y expr={pow(10, \thisrow{W2(method3)_mu} - \thisrow{MLE_mu})}] {\vMn};
			\addlegendentry{W2/MLE}
			\addplot[thin, color=black, dotted,] {1};
		\end{axis}
	\end{tikzpicture}
	\begin{tikzpicture}
		\tikzset{every node/.style={}}
		\begin{axis}
			[width=6cm,
			xmax=5,xmin=2,
			ymax=1.15,ymin=0.98,
			xlabel={$\log_{10} n$},
			yticklabel style={
				/pgf/number format/.cd,
				fixed,
				zerofill,
				precision=3,
			},
			ylabel near ticks,
			legend pos=north east,
			legend style={legend cell align=left,draw=none,fill=white,fill opacity=0.8,text opacity=1,},
			]
			\addplot[color=black, mark=square*, mark options={solid}] table [x=log10N, y expr={pow(10, \thisrow{W1(method2)_kappa} - \thisrow{MLE_kappa})}] {\vMn};
			\addlegendentry{W1/MLE}
			\addplot[color=black, dashed, mark=*, mark options={solid}] table [x=log10N, y expr={pow(10, \thisrow{W2(method3)_kappa} - \thisrow{MLE_kappa})}] {\vMn};
			\addlegendentry{W2/MLE}
			\addplot[color=black, dotted,] {1};
		\end{axis}
	\end{tikzpicture}
	\caption{Ratio of mean squared error of the $L^1$ and $L^2$ Wasserstein projection estimators to that of the maximum likelihood estimator for the von Mises distribution ($\mu=0.3,\kappa=2$). Left: $\mu$, Right: $\kappa$.}
	\label{fig_vM}
\end{figure}
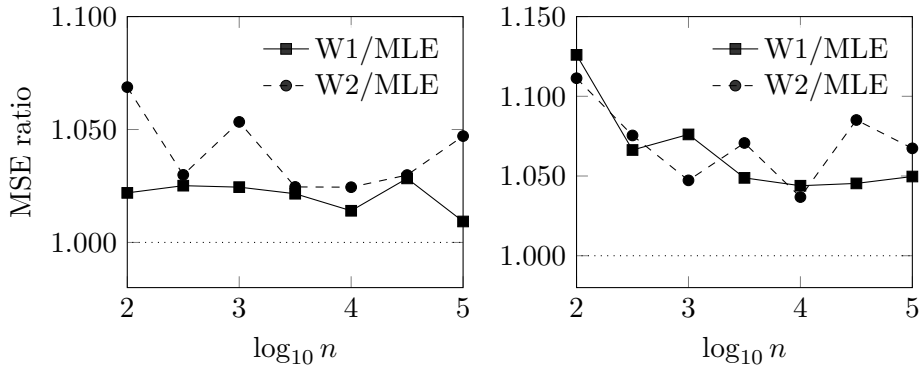

\begin{figure}[htbp]
	\centering
	\begin{tikzpicture}
		\tikzset{every node/.style={}}
		\begin{axis}
			[width=6cm,
			xmode=log,
			xmax=500,xmin=0.5,
			ymax=1.1,
			xlabel={$\kappa$},
			ylabel={MSE ratio},
			yticklabel style={
				/pgf/number format/.cd,
				fixed,
				zerofill,
				precision=3,
			},
			ylabel near ticks,
			legend pos=north east,
			legend style={legend cell align=left,draw=none,fill=white,fill opacity=0.8,text opacity=1,},
			]
			\addplot[color=black, mark=square*, mark options={solid}] table [x=kappa, y=W1(method2)_mu_MLE_mu] {\vMk};
			\addlegendentry{W1/MLE}
			\addplot[color=black, dashed, mark=*, mark options={solid}] table [x=kappa, y=W2(method3)_mu_MLE_mu] {\vMk};
			\addlegendentry{W2/MLE}
			\addplot[thin, color=black, dotted, domain=0.5:500, samples=500] {1};
		\end{axis}
	\end{tikzpicture}
	\begin{tikzpicture}
		\tikzset{every node/.style={}}
		\begin{axis}
			[width=6cm,
			xmode=log,
			xmax=500,xmin=0.5,
			ymax=1.4,
			xlabel={$\kappa$},
			yticklabel style={
				/pgf/number format/.cd,
				fixed,
				zerofill,
				precision=3,
			},
			ylabel near ticks,
			legend pos=north east,
			legend style={legend cell align=left,draw=none,fill=white,fill opacity=0.8,text opacity=1,},
			]
			\addplot[color=black, mark=square*, mark options={solid}] table [x=kappa, y=W1(method2)_kappa_MLE_kappa] {\vMk};
			\addlegendentry{W1/MLE}
			\addplot[color=black, dashed, mark=*, mark options={solid}] table [x=kappa, y=W2(method3)_kappa_MLE_kappa] {\vMk};
			\addlegendentry{W2/MLE}
			\addplot[color=black, dotted, domain=0.5:500] {1};
		\end{axis}
	\end{tikzpicture}
	\caption{Ratio of mean squared error of the $L^1$ and $L^2$ Wasserstein projection estimators to that of the maximum likelihood estimator for the von Mises distribution ($\mu=0.3,n=10^5$). Left: $\mu$, Right: $\kappa$.}
	\label{fig_vM2}
\end{figure}
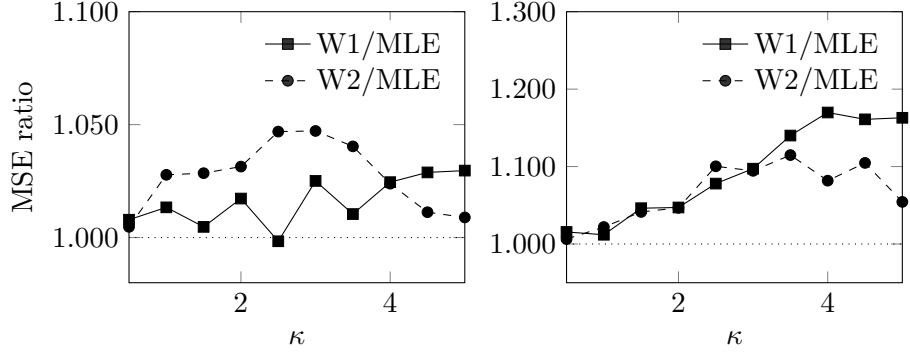

\subsection{Wrapped Cauchy distribution}
Here, we consider the wrapped Cauchy distribution defined by
\[
p(x \mid \mu,\rho) = \frac{1}{2\pi} \frac{1-\rho^2}{1+\rho^2-2 \rho \cos(x-\mu)},
\]
where $\mu \in [0, 2\pi)$ is the mean paramter and $\rho \in [0,1)$ is the concentration parameter.
Its Fisher information matrix is
\[
I(\mu,\rho) = \begin{pmatrix} \frac{2 \rho^2}{(1-\rho^2)^2} & 0 \\ 0 & \frac{2}{(1-\rho^2)^2} \end{pmatrix}.
\]
Several methods have been developed for computing the maximum likelihood estimator \cite{Kent,McCullagh,Okamura}.

Figures~\ref{fig_wC} and \ref{fig_wC2} display the mean squared error (MSE) ratios of the $L^1$ and $L^2$ Wasserstein projection estimators to the maximum likelihood estimator, given by
\[
\frac{{\rm E}[(\hat{\mu}_{W,p}-\mu)^2]}{{\rm E}[(\hat{\mu}_{\mathrm{ML}}-\mu)^2]} \quad \text{and} \quad \frac{{\rm E}[(\hat{\rho}_{W,p}-\rho)^2]}{{\rm E}[(\hat{\rho}_{\mathrm{ML}}-\rho)^2]}.
\]
Specifically, the ratios are plotted against the sample size $n$ in Figure~\ref{fig_wC}, and against the concentration parameter $\rho$ in Figure~\ref{fig_wC2}.
Note that the mean squared errors do not depend on $\mu$ due to shift invariance.
Again, the Wasserstein projection estimators attain estimation accuracy comparable to the maximum likelihood estimator, except when $\rho$ is very close to one.

\begin{figure}[htbp]
	\centering
	\begin{tikzpicture}
		\tikzset{every node/.style={}}
		\begin{axis}
			[width=6cm,
			xmax=5,xmin=2,
			ymax=1.5,ymin=0.9,
			xlabel={$\log_{10} n$},
			ylabel={MSE ratio},
			yticklabel style={
				/pgf/number format/.cd,
				fixed,
				zerofill,
				precision=3,
			},
			ylabel near ticks,
			legend pos=north east,
			legend style={legend cell align=left,draw=none,fill=white,fill opacity=0.8,text opacity=1,},
			]
			\addplot[color=black, mark=square*, mark options={solid}] table [x=log10N, y expr={pow(10, \thisrow{W1(method2)_mu} - \thisrow{MLE_mu})}] {\wCn};
			\addlegendentry{W1/MLE}
			\addplot[color=black, dashed, mark=*, mark options={solid}] table [x=log10N, y expr={pow(10, \thisrow{W2(method3)_mu} - \thisrow{MLE_mu})}] {\wCn};
			\addlegendentry{W2/MLE}
			\addplot[thin, color=black, dotted,] {1};
		\end{axis}
	\end{tikzpicture}
	\begin{tikzpicture}
		\tikzset{every node/.style={}}
		\begin{axis}
			[width=6cm,
			xmax=5,xmin=2,
			ymax=1.5,ymin=0.9,
			xlabel={$\log_{10} n$},
			yticklabel style={
				/pgf/number format/.cd,
				fixed,
				zerofill,
				precision=3,
			},
			ylabel near ticks,
			legend pos=north east,
			legend style={legend cell align=left,draw=none,fill=white,fill opacity=0.8,text opacity=1,},
			]
			\addplot[color=black, mark=square*, mark options={solid}] table [x=log10N, y expr={pow(10, \thisrow{W1(method2)_rho} - \thisrow{MLE_rho})}] {\wCn};
			\addlegendentry{W1/MLE}
			\addplot[color=black, dashed, mark=*, mark options={solid}] table [x=log10N, y expr={pow(10, \thisrow{W2(method3)_rho} - \thisrow{MLE_rho})}] {\wCn};
			\addlegendentry{W2/MLE}
			\addplot[color=black, dotted,] {1};
		\end{axis}
	\end{tikzpicture}
	\caption{Ratio of mean squared error of the $L^1$ and $L^2$ Wasserstein projection estimators to that of the maximum likelihood estimator for the wrapped Cauchy distribution ($\mu=\pi/8,\rho=0.4$). Left: $\mu$, Right: $\rho$.}
	\label{fig_wC}
\end{figure}
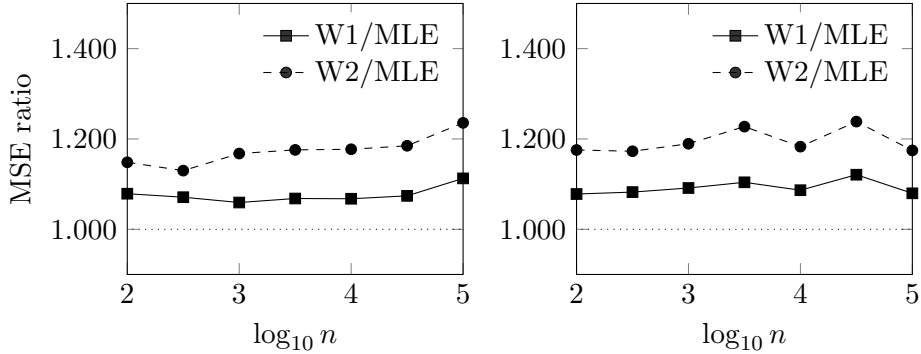

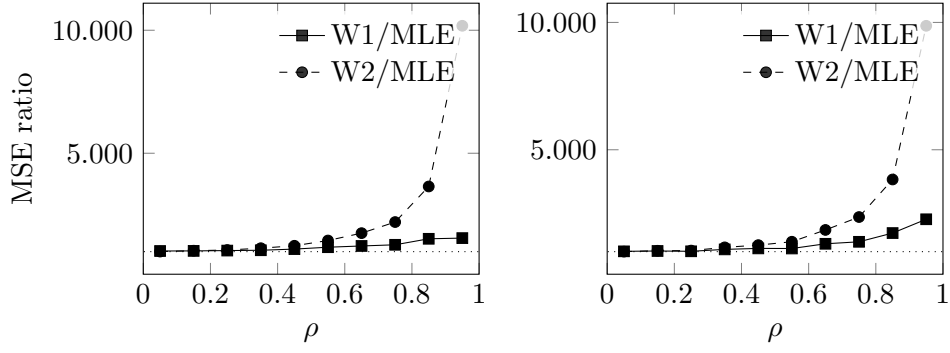
\begin{figure}[htbp]
	\centering
	\begin{tikzpicture}
		\tikzset{every node/.style={}}
		\begin{axis}
			[width=6cm,
			xmax=1,xmin=0,
			xlabel={$\rho$},
			ylabel={MSE ratio},
			yticklabel style={
				/pgf/number format/.cd,
				fixed,
				zerofill,
				precision=3,
			},
			ylabel near ticks,
			legend pos=north east,
			legend style={legend cell align=left,draw=none,fill=white,fill opacity=0.8,text opacity=1,},
			]
			\addplot[color=black, mark=square*, mark options={solid}] table [x=rho, y =W1(method2)_mu_MLE_mu] {\wCrho};
			\addlegendentry{W1/MLE}
			\addplot[color=black, dashed, mark=*, mark options={solid}] table [x=rho, y=W2(method3)_mu_MLE_mu] {\wCrho};
			\addlegendentry{W2/MLE}
			\addplot[thin, color=black, dotted,] {1};
		\end{axis}
	\end{tikzpicture}
	\begin{tikzpicture}
		\tikzset{every node/.style={}}
		\begin{axis}
			[width=6cm,
			xmax=1,xmin=0,
			xlabel={$\rho$},
			yticklabel style={
				/pgf/number format/.cd,
				fixed,
				zerofill,
				precision=3,
			},
			ylabel near ticks,
			legend pos=north east,
			legend style={legend cell align=left,draw=none,fill=white,fill opacity=0.8,text opacity=1,},
			]
			\addplot[color=black, mark=square*, mark options={solid}] table [x=rho, y=W1(method2)_rho_MLE_rho] {\wCrho};
			\addlegendentry{W1/MLE}
			\addplot[color=black, dashed, mark=*, mark options={solid}] table [x=rho, y=W2(method3)_rho_MLE_rho] {\wCrho};
			\addlegendentry{W2/MLE}
			\addplot[color=black, dotted,] {1};
		\end{axis}
	\end{tikzpicture}
	\caption{Ratio of mean squared error of the $L^1$ and $L^2$ Wasserstein projection estimators to that of the maximum likelihood estimator for the wrapped Cauchy distribution ($\mu=\pi/8,n=10^5$). Left: $\mu$, Right: $\rho$.}
	\label{fig_wC2}
\end{figure}

\subsection{Sine-skewed von Mises distribution}
Here, we consider the sine-skewed von Mises distribution \cite{skewed} defined by
\[
p(x \mid \mu, \kappa, \lambda) = \frac{1}{2\pi I_{0}(\kappa)} \exp (\kappa \cos(x - \mu) ) (1 + \lambda \sin(x - \mu)),
\]
where $\mu \in [0,2\pi)$ is the mean parameter, $\kappa>0$ is the concentration parameter and $\lambda \in [-1,1]$ is the skewness parameter.
Its Fisher information is
\begin{align*}
	I_{\mu \mu} (\mu, \kappa, \lambda) &= \kappa \dfrac{I_{1}(\kappa)}{I_{0}(\kappa)}
	+ \dfrac{\lambda}{2\pi I_{0}(\kappa)} \int_{-\pi}^{\pi} \exp (\kappa \cos x)
	\frac{\lambda + \sin x}{1 + \lambda \sin x} {\rm d} x, \\
	I_{\mu \kappa} (\mu, \kappa, \lambda) &=  \frac{\lambda}{2} \left(\frac{I_{2}(\kappa)}{I_{0}(\kappa)} - 1 \right), \\
	I_{\mu \lambda} (\mu, \kappa, \lambda) &= \frac{1}{2\pi I_{0}(\kappa)} \int_{-\pi}^{\pi} \exp (\kappa \cos x)
	\frac{\cos x}{1+\lambda \sin x} {\rm d} x, \\
	I_{\kappa \kappa} (\mu, \kappa, \lambda) &= \frac{1}{2} + \frac{I_{2}(\kappa)}{2 I_{0}(\kappa)} - 
	\left(\frac{I_{1}(\kappa)}{I_{0}(\kappa)}\right)^{2}, \\
	I_{\kappa \lambda} (\mu, \kappa, \lambda) &=  0, \\
	I_{\lambda \lambda} (\mu, \kappa, \lambda) &= \frac{1}{2\pi I_{0}(\kappa)} \int_{-\pi}^{\pi} \exp (\kappa \cos x)
	\frac{\sin^{2} x}{1+\lambda \sin x} {\rm d} x.
\end{align*}

Figures~\ref{fig_skew} and \ref{fig_skew2} display the mean squared error (MSE) ratios of the $L^1$ Wasserstein projection estimators to the maximum likelihood estimator, given by
\[
\frac{{\rm E}[(\hat{\mu}_{W,p}-\mu)^2]}{{\rm E}[(\hat{\mu}_{\mathrm{ML}}-\mu)^2]}, \quad \frac{{\rm E}[(\hat{\kappa}_{W,p}-\kappa)^2]}{{\rm E}[(\hat{\kappa}_{\mathrm{ML}}-\kappa)^2]}, \quad \text{and} \quad \frac{{\rm E}[(\hat{\lambda}_{W,p}-\lambda)^2]}{{\rm E}[(\hat{\lambda}_{\mathrm{ML}}-\lambda)^2]}.
\]
Specifically, the ratios are plotted against the sample size $n$ in Figure~\ref{fig_skew}, and against the skewness parameter $\lambda$ in Figure~\ref{fig_skew2}.
Note that the mean squared errors do not depend on $\mu$ due to shift invariance.
The mean squared error of the Wasserstein projection estimators is at most three times that of the maximum likelihood estimator.

\begin{figure}[htbp]
	\centering
	\begin{tikzpicture}
		\tikzset{every node/.style={}}
		\begin{axis}
			[
			xmax=5,xmin=2,ymax=4,
			width=6cm,
			xlabel={$\log_{10} n$},
			ylabel={MSE ratio},
			yticklabel style={
				/pgf/number format/.cd,
				fixed,
				zerofill,
				precision=3,
			},
			ylabel near ticks,
			legend pos=north east,
			legend style={legend cell align=left,draw=none,fill=white,fill opacity=0.8,text opacity=1,},
			]
			\addplot[color=black, mark=square*, mark options={solid}] table [x=log10N, y expr={pow(10, \thisrow{W1(method2)_mu} - \thisrow{MLE_mu})}] {\svMn};
			\addlegendentry{W1/MLE}
			\addplot[thin, color=black, dotted,] {1};
		\end{axis}
	\end{tikzpicture} 
	\begin{tikzpicture}
		\tikzset{every node/.style={}}
		\begin{axis}
			[
			xmax=5,xmin=2,ymax=3,
			width=6cm,
			xlabel={$\log_{10} n$},
			yticklabel style={
				/pgf/number format/.cd,
				fixed,
				zerofill,
				precision=3,
			},
			ylabel near ticks,
			legend pos=north east,
			legend style={legend cell align=left,draw=none,fill=white,fill opacity=0.8,text opacity=1,},
			]
			\addplot[color=black, mark=square*, mark options={solid}] table [x=log10N, y expr={pow(10, \thisrow{W1(method2)_kappa} - \thisrow{MLE_kappa})}] {\svMn};
			\addlegendentry{W1/MLE}
			\addplot[color=black, dotted,] {1};
		\end{axis}
	\end{tikzpicture}
	\begin{tikzpicture}
		\tikzset{every node/.style={}}
		\begin{axis}
			[
			xmax=5,xmin=2,ymax=5,
			width=6cm,
			xlabel={$\log_{10} n$},
			yticklabel style={
				/pgf/number format/.cd,
				fixed,
				zerofill,
				precision=3,
			},
			ylabel near ticks,
			legend pos=north east,
			legend style={legend cell align=left,draw=none,fill=white,fill opacity=0.8,text opacity=1,},
			]
			\addplot[color=black, mark=square*, mark options={solid}] table [x=log10N, y expr={pow(10, \thisrow{W1(method2)_lambda} - \thisrow{MLE_lambda})}] {\svMn};
			\addlegendentry{W1/MLE}
			\addplot[color=black, dotted,] {1};
		\end{axis}
	\end{tikzpicture} 
	\caption{Mean squared error of the $L^1$ Wasserstein projection estimator and maximum likelihood estimator for the sine-skewed von Mises distribution ($\mu=0,\kappa=1,\lambda=0.7$). Upper left: $\mu$, Upper right: $\kappa$, Lower: $\lambda$.}
	\label{fig_skew}
\end{figure}
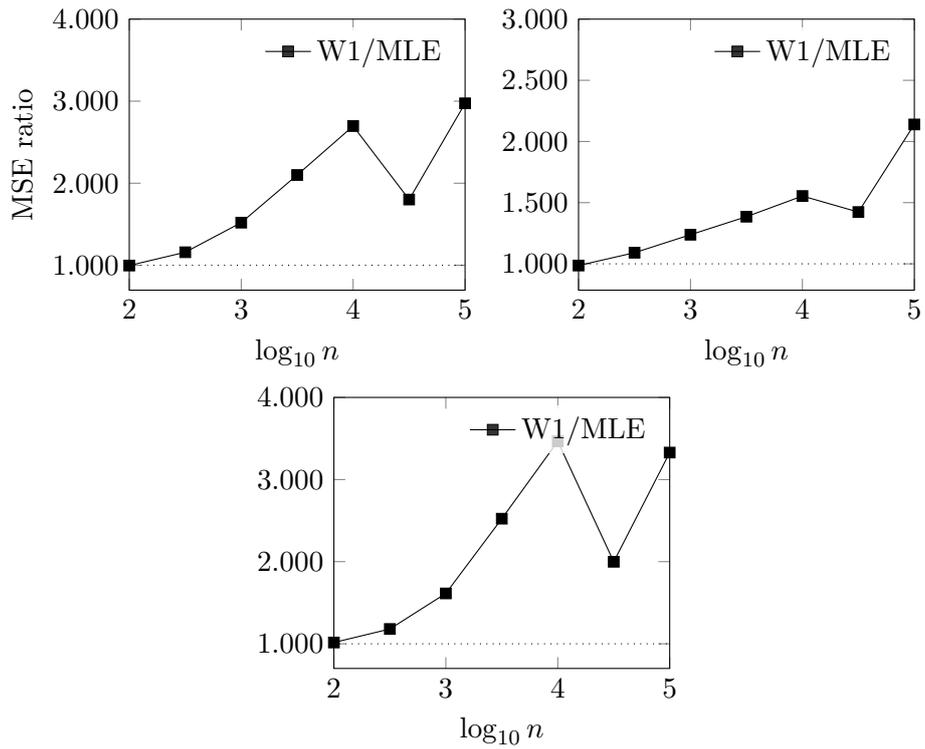

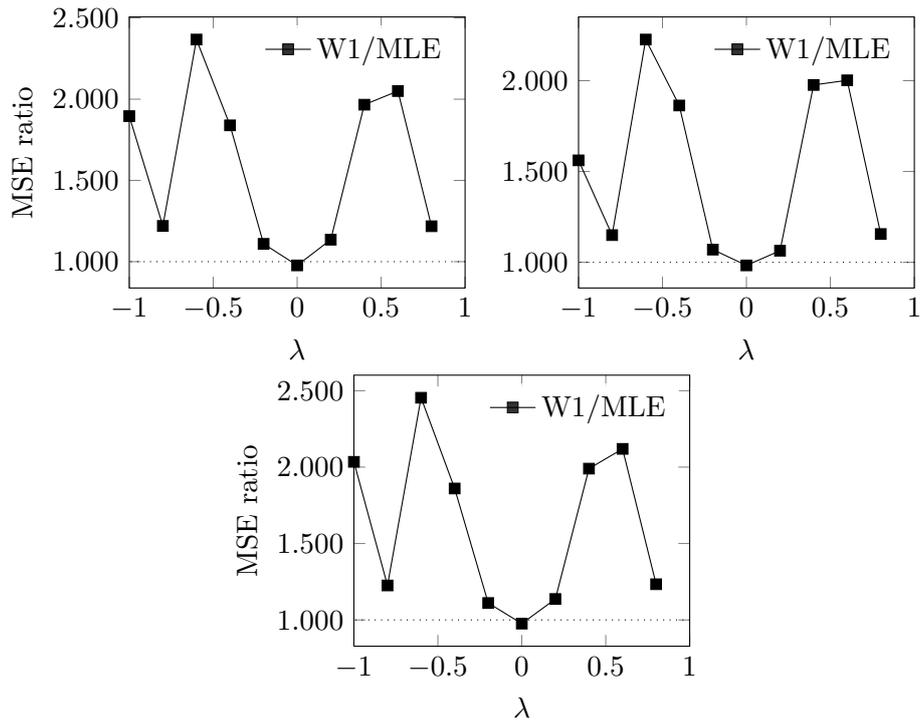
\begin{figure}[htbp]
	\centering
	\begin{tikzpicture}
		\tikzset{every node/.style={}}
		\begin{axis}
			[
			xmax=1,xmin=-1,
			width=6cm,
			xlabel={$\lambda$},
			ylabel={MSE ratio},
			yticklabel style={
				/pgf/number format/.cd,
				fixed,
				zerofill,
				precision=3,
			},
			ylabel near ticks,
			legend pos=north east,
			legend style={legend cell align=left,draw=none,fill=white,fill opacity=0.8,text opacity=1,},
			]
			\addplot[color=black, mark=square*, mark options={solid}] table [x=lambda, y expr={pow(10, \thisrow{W1(method2)_mu} - \thisrow{MLE_mu})}] {\svMlambda};
			\addlegendentry{W1/MLE}
			\addplot[thin, color=black, dotted,] {1};
		\end{axis}
	\end{tikzpicture}
	\begin{tikzpicture}
		\tikzset{every node/.style={}}
		\begin{axis}
			[
			xmax=1,xmin=-1,
			width=6cm,
			xlabel={$\lambda$},
			yticklabel style={
				/pgf/number format/.cd,
				fixed,
				zerofill,
				precision=3,
			},
			ylabel near ticks,
			legend pos=north east,
			legend style={legend cell align=left,draw=none,fill=white,fill opacity=0.8,text opacity=1,},
			]
			\addplot[color=black, mark=square*, mark options={solid}] table [x=lambda, y expr={pow(10, \thisrow{W1(method2)_kappa} - \thisrow{MLE_kappa})}] {\svMlambda};
			\addlegendentry{W1/MLE}
			\addplot[color=black, dotted,] {1};
		\end{axis}
	\end{tikzpicture}
	\begin{tikzpicture}
		\tikzset{every node/.style={}}
		\begin{axis}
			[
			xmax=1,xmin=-1,
			width=6cm,
			xlabel={$\lambda$},
			ylabel={MSE ratio},
			yticklabel style={
				/pgf/number format/.cd,
				fixed,
				zerofill,
				precision=3,
			},
			ylabel near ticks,
			legend pos=north east,
			legend style={legend cell align=left,draw=none,fill=white,fill opacity=0.8,text opacity=1,},
			]
			\addplot[color=black, mark=square*, mark options={solid}] table [x=lambda, y expr={pow(10, \thisrow{W1(method2)_lambda} - \thisrow{MLE_lambda})}] {\svMlambda};
			\addlegendentry{W1/MLE}
			\addplot[color=black, dotted,] {1};
		\end{axis}
	\end{tikzpicture}
	\caption{Mean squared error of the $L^1$ Wasserstein projection estimator and maximum likelihood estimator for the sine-skewed von Mises distribution ($\mu=0,\kappa=1,n=10^5$). Upper left: $\mu$, Upper right: $\kappa$, Lower: $\lambda$.}
	\label{fig_skew2}
\end{figure}

\subsection{Robustness against noise contamination}\label{sec_contam}
Here, we study the robustness of the Wasserstein projection estimators against noise contamination.
We generated samples from the von Mises distribution with uniform noise contamination: 
\[
p(x \mid \mu,\kappa,\varepsilon) = (1-\varepsilon) \frac{1}{2 \pi I_0(\kappa)}\exp (\kappa \cos (\theta-\mu)) + \varepsilon \frac{1}{2 \pi},
\]
where $\varepsilon \in [0,1]$ is the noise contamination ratio.
Namely, each sample comes from the uniform distribution on $[0,2\pi)$ with probability $\varepsilon$.
In this setting, we consider fitting the von Mises distribution (without noise contamination) and compare the performance of the Wasserstein projection estimators and the maximum likelihood estimator.

Figures~\ref{fig_robust} and \ref{fig_robust2} plot the ratio of the mean squared errors of the Wasserstein projection estimators to that of the maximum likelihood estimator with respect to $n$ and $\rho$, respectively.
In Figure~\ref{fig_robust}, we also plot two estimators defined by projection with respect to robust divergences \cite{Kato} \footnote{We set the hyperparameters based on the numerical results reported in \cite{Kato}. Therefore, the resulting estimators should be regarded as oracle estimators and are not applicable in practice. Note that the Wasserstein estimators do not require such a hyperparameter tuning.}.
Note that the mean squared errors do not depend on $\mu$ due to shift invariance.
In this case, the maximum likelihood estimator has negative bias in estimating $\kappa$ due to noise, which leads to limited estimation accuracy.
On the other hand, the $L^1$ Wasserstein projection estimator attains better estimation accuracy for $\kappa$ than the maximum likelihood estimator.
This result implies the robustness of the $L^1$ Wasserstein projection estimator to noise contamination, which is compatible to the usage of $l_1$ loss in robust statistics \cite{Huber}.

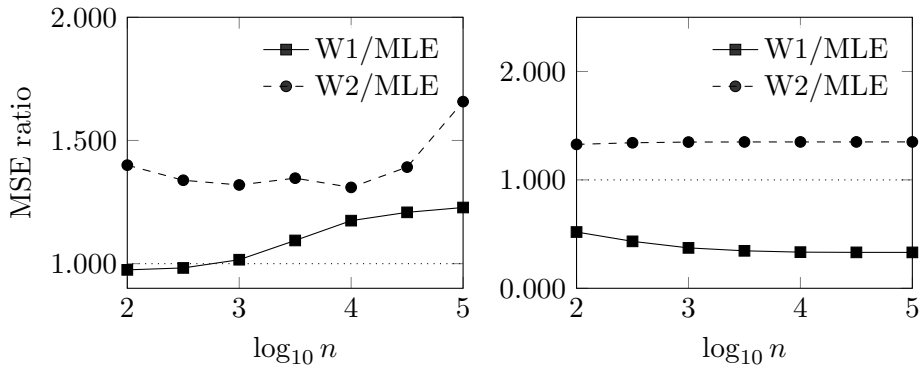
\begin{figure}[htbp]
	\centering
	\begin{tikzpicture}
		\tikzset{every node/.style={}}
		\begin{axis}
			[width=6cm,
			xmax=5,xmin=2,
			ymax=3,ymin=0.7,
			xlabel={$\log_{10} n$},
			ylabel={MSE ratio},
			yticklabel style={
				/pgf/number format/.cd,
				fixed,
				zerofill,
				precision=3,
			},
			ylabel near ticks,
			legend pos=north east,
			legend style={legend cell align=left,draw=none,fill=white,fill opacity=0.8,text opacity=1,},
			]
			\addplot[color=black, mark=square*, mark options={solid}] table [x=log10N, y expr={pow(10, \thisrow{W1(method2)_mu} - \thisrow{MLE_mu})}] {\robust};
			\addlegendentry{W1/MLE}
			\addplot[color=black, dashed, mark=*, mark options={solid}] table [x=log10N, y expr={pow(10, \thisrow{W2(method3)_mu} - \thisrow{MLE_mu})}] {\robust};
			\addlegendentry{W2/MLE}
			\addplot[color=black, dotted, mark=triangle*, mark options={solid}] table [x=log10N, y expr={pow(10, \thisrow{type1_beta05_mu} - \thisrow{MLE_mu})}] {\robust};
			\addlegendentry{beta/MLE}
			\addplot[color=black, dashdotted, mark=diamond*, mark options={solid}] table [x=log10N, y expr={pow(10, \thisrow{type0_gamma05_mu} - \thisrow{MLE_mu})}] {\robust};
			\addlegendentry{gamma/MLE}
			\addplot[thin, color=black, dotted,] {1};
		\end{axis}
	\end{tikzpicture}
	\begin{tikzpicture}
		\tikzset{every node/.style={}}
		\begin{axis}
			[width=6cm,
			xmax=5,xmin=2,
			ymax=5,ymin=0,
			xlabel={$\log_{10} n$},
			yticklabel style={
				/pgf/number format/.cd,
				fixed,
				zerofill,
				precision=3,
			},
			ylabel near ticks,
			legend pos=north east,
			legend style={legend cell align=left,draw=none,fill=white,fill opacity=0.8,text opacity=1,},
			]
			\addplot[color=black, mark=square*, mark options={solid}] table [x=log10N, y expr={pow(10, \thisrow{W1(method2)_kappa} - \thisrow{MLE_kappa})}] {\robust};
			\addlegendentry{W1/MLE}
			\addplot[color=black, dashed, mark=*, mark options={solid}] table [x=log10N, y expr={pow(10, \thisrow{W2(method3)_kappa} - \thisrow{MLE_kappa})}] {\robust};
			\addlegendentry{W2/MLE}
			\addplot[color=black, dashdotted, mark=triangle*, mark options={solid}] table [x=log10N, y expr={pow(10, \thisrow{type1_beta05_kappa} - \thisrow{MLE_kappa})}] {\robust};
			\addlegendentry{beta/MLE}
			\addplot[color=black, dotted, mark=diamond*, mark options={solid}] table [x=log10N, y expr={pow(10, \thisrow{type0_gamma05_kappa} - \thisrow{MLE_kappa})}] {\robust};
			\addlegendentry{gamma/MLE}
			\addplot[color=black, dotted,] {1};
		\end{axis}
	\end{tikzpicture}
	\caption{Ratio of mean squared error of the $L^1$ and $L^2$ Wasserstein projection estimators and two robust estimators by \cite{Kato} (beta and gamma) to that of the maximum likelihood estimator for the von Mises distribution with noise contamination ($\mu=\pi/4, \kappa=5,\varepsilon=0.1$). Left: $\mu$, Right: $\kappa$.}
	\label{fig_robust}
\end{figure}

\begin{figure}[htbp]
	\centering
	\begin{tikzpicture}
		\tikzset{every node/.style={}}
		\begin{axis}
			[width=6cm,
			xmax=0.2,xmin=0,
			ymax=2.5,ymin=0.95,
			xlabel={$\epsilon$},
			ylabel={MSE ratio},
			yticklabel style={
				/pgf/number format/.cd,
				fixed,
				zerofill,
				precision=3,
			},
			ylabel near ticks,
			legend pos=north east,
			legend style={legend cell align=left,draw=none,fill=white,fill opacity=0.8,text opacity=1,},
			]
			\addplot[color=black, mark=square*, mark options={solid}] table [x=noise_rate, y expr={pow(10, \thisrow{W1(method2)_mu} - \thisrow{MLE_mu})}] {\robustt};
			\addlegendentry{W1/MLE}
			\addplot[color=black, dashed, mark=*, mark options={solid}] table [x=noise_rate, y expr={pow(10, \thisrow{W2(method3)_mu} - \thisrow{MLE_mu})}] {\robustt};
			\addlegendentry{W2/MLE}
			\addplot[thin, color=black, dotted,] {1};
		\end{axis}
	\end{tikzpicture}
	\begin{tikzpicture}
		\tikzset{every node/.style={}}
		\begin{axis}
			[width=6cm,
			xmax=0.2,xmin=0,
			ymax=2.5,ymin=0,
			xlabel={$\epsilon$},
			yticklabel style={
				/pgf/number format/.cd,
				fixed,
				zerofill,
				precision=3,
			},
			ylabel near ticks,
			legend pos=north east,
			legend style={legend cell align=left,draw=none,fill=white,fill opacity=0.8,text opacity=1,},
			]
			\addplot[color=black, mark=square*, mark options={solid}] table [x=noise_rate, y expr={pow(10, \thisrow{W1(method2)_kappa} - \thisrow{MLE_kappa})}] {\robustt};
			\addlegendentry{W1/MLE}
			\addplot[color=black, dashed, mark=*, mark options={solid}] table [x=noise_rate, y expr={pow(10, \thisrow{W2(method3)_kappa} - \thisrow{MLE_kappa})}] {\robustt};
			\addlegendentry{W2/MLE}
			\addplot[color=black, dotted,] {1};
		\end{axis}
	\end{tikzpicture}
	\caption{Ratio of mean squared error of the $L^1$ and $L^2$ Wasserstein projection estimators to that of the maximum likelihood estimator for the von Mises distribution with noise contamination ($\mu=\pi/4, \kappa=5, n=10^5$). Left: $\mu$, Right: $\kappa$.}
	\label{fig_robust2}
\end{figure}

\subsection{Performance under model misspecification}\label{sec_mis}
Finally, we study the performance of the Wasserstein projection estimators under model misspecification. 
Figures~\ref{fig_vMWC} and \ref{fig_WCvM} compare the Wasserstein projection estimators and the maximum likelihood estimator when the wrapped Cauchy distribution is fitted to samples from the von Mises distribution and vice versa, respectively. 
The estimation accuracy is evaluated by the mean squared error of the circular mean parameter $\mu$, as well as the Kullback--Leibler divergence and the Wasserstein distances ($p=1$ or $p=2$) between the true and plug-in densities. 
See the Appendix for the analytical formulas of the Kullback--Leibler divergence between the von Mises and wrapped Cauchy distributions.

For small $\kappa$ and $\rho$, the distributions are close to uniform, and the three estimators show comparable performance across all metrics. On the other hand, for large $\kappa$ and $\rho$, the distributions become highly peaked, and the three estimators yield distinct results. This discrepancy is inherently tied to the structural differences in their tails, specifically the much heavier tails of the wrapped Cauchy distribution compared to the von Mises distribution. This points to a fundamental peak-tail trade-off: the maximum likelihood estimator prioritizes covering the tails, whereas the $L^1$ Wasserstein projection estimator focuses on matching the peak where the majority of the mass is concentrated. A comprehensive theoretical characterization of this behavior remains an interesting direction for future research.

Note that the simulation results in Sections~\ref{sec_contam} and \ref{sec_mis} should be considered as a starting point for a systematic study of the stability and robustness of Wasserstein projection estimators.

\begin{figure}
	\begin{tikzpicture}
		\tikzset{every node/.style={}}
		\begin{axis}
			[width=6cm,
			xmode=log,
			ymode=log,
			xmax=500,xmin=0.5,
			xlabel={$\kappa$},
			ylabel={MSE of $\mu$},
			ylabel near ticks,
			legend pos=south west,
			legend style={legend cell align=left,draw=none,fill=white,fill opacity=0.8,text opacity=1,},
			]
			\addplot[color=black, mark=square*] table [x=kappa, y=MLE_mu_mse] {\vMWC};
			\addlegendentry{MLE}
			\addplot[color=black, dashed, mark=*] table [x=kappa, y=W1(method2)_mu_mse] {\vMWC};
			\addlegendentry{W1}
			\addplot[color=black, dotted, mark=triangle*] table [x=kappa, y=W2(method3)_mu_mse] {\vMWC};
			\addlegendentry{W2}
		\end{axis}
	\end{tikzpicture}
	\begin{tikzpicture}
		\tikzset{every node/.style={}}
		\begin{axis}
			[width=6cm,
			xmode=log,
			ymode=log,
			xmax=500,xmin=0.5,
			xlabel={$\kappa$},
			ylabel={KL},
			ylabel near ticks,
			legend pos=south east,
			legend style={legend cell align=left,draw=none,fill=white,fill opacity=0.8,text opacity=1,},
			]
			\addplot[color=black, mark=square*] table [x=kappa, y=MLE_KL] {\vMWC};
			\addlegendentry{MLE}
			\addplot[color=black, dashed, mark=*] table [x=kappa, y=W1(method2)_KL] {\vMWC};
			\addlegendentry{W1}
			\addplot[color=black, dotted, mark=triangle*] table [x=kappa, y=W2(method3)_KL] {\vMWC};
			\addlegendentry{W2}
		\end{axis}
	\end{tikzpicture}\\
	\begin{tikzpicture}
		\tikzset{every node/.style={}}
		\begin{axis}
			[width=6cm,
			xmode=log,
			ymode=log,
			xmax=500,xmin=0.5,
			xlabel={$\kappa$},
			ylabel={$W_1$},
			ylabel near ticks,
			legend pos=south east,
			legend style={legend cell align=left,draw=none,fill=white,fill opacity=0.8,text opacity=1,},
			]
			\addplot[color=black, mark=square*] table [x=kappa, y=MLE_W1] {\vMWC};
			\addlegendentry{MLE}
			\addplot[color=black, dashed, mark=*] table [x=kappa, y=W1(method2)_W1] {\vMWC};
			\addlegendentry{W1}
			\addplot[color=black, dotted, mark=triangle*] table [x=kappa, y=W2(method3)_W1] {\vMWC};
			\addlegendentry{W2}
		\end{axis}
	\end{tikzpicture}
	\begin{tikzpicture}
		\tikzset{every node/.style={}}
		\begin{axis}
			[width=6cm,
			xmode=log,
			ymode=log,
			xmax=500,xmin=0.5,
			xlabel={$\kappa$},
			ylabel={$W_2$},
			ylabel near ticks,
			legend pos=south east,
			legend style={legend cell align=left,draw=none,fill=white,fill opacity=0.8,text opacity=1,},
			]
			\addplot[color=black, mark=square*] table [x=kappa, y=MLE_W2] {\vMWC};
			\addlegendentry{MLE}
			\addplot[color=black, dashed, mark=*] table [x=kappa, y=W1(method2)_W2] {\vMWC};
			\addlegendentry{W1}
			\addplot[color=black, dotted, mark=triangle*] table [x=kappa, y=W2(method3)_W2] {\vMWC};
			\addlegendentry{W2}
		\end{axis}
	\end{tikzpicture}
	\caption{Comparison of estimators for the wrapped Cauchy distribution under samples from the von Mises distribution.}
	\label{fig_vMWC}
\end{figure}

\begin{figure}
	\begin{tikzpicture}
		\tikzset{every node/.style={}}
		\begin{axis}
			[width=6cm,
			ymode=log,
			xmax=0.95,xmin=0.05,
			xlabel={$\rho$},
			ylabel={MSE of $\mu$},
			ylabel near ticks,
			legend pos=south west,
			legend style={legend cell align=left,draw=none,fill=white,fill opacity=0.8,text opacity=1,},
			]
			\addplot[color=black, mark=square*] table [x=rho, y=MLE_mu_mse] {\WCvM};
			\addlegendentry{MLE}
			\addplot[color=black, dashed, mark=*] table [x=rho, y=W1(method2)_mu_mse] {\WCvM};
			\addlegendentry{W1}
			\addplot[color=black, dotted, mark=triangle*] table [x=rho, y=W2(method3)_mu_mse] {\WCvM};
			\addlegendentry{W2}
		\end{axis}
	\end{tikzpicture}
	\begin{tikzpicture}
		\tikzset{every node/.style={}}
		\begin{axis}
			[width=6cm,
			ymode=log,
			xmax=0.95,xmin=0.05,
			xlabel={$\rho$},
			ylabel={KL},
			ylabel near ticks,
			legend pos=south east,
			legend style={legend cell align=left,draw=none,fill=white,fill opacity=0.8,text opacity=1,},
			]
			\addplot[color=black, mark=square*] table [x=rho, y=MLE_KL] {\WCvM};
			\addlegendentry{MLE}
			\addplot[color=black, dashed, mark=*] table [x=rho, y=W1(method2)_KL] {\WCvM};
			\addlegendentry{W1}
			\addplot[color=black, dotted, mark=triangle*] table [x=rho, y=W2(method3)_KL] {\WCvM};
			\addlegendentry{W2}
		\end{axis}
	\end{tikzpicture}\\
	\begin{tikzpicture}
		\tikzset{every node/.style={}}
		\begin{axis}
			[width=6cm,
			ymode=log,
			xmax=0.95,xmin=0.05,
			xlabel={$\rho$},
			ylabel={$W_1$},
			ylabel near ticks,
			legend pos=south east,
			legend style={legend cell align=left,draw=none,fill=white,fill opacity=0.8,text opacity=1,},
			]
			\addplot[color=black, mark=square*] table [x=rho, y=MLE_W1] {\WCvM};
			\addlegendentry{MLE}
			\addplot[color=black, dashed, mark=*] table [x=rho, y=W1(method2)_W1] {\WCvM};
			\addlegendentry{W1}
			\addplot[color=black, dotted, mark=triangle*] table [x=rho, y=W2(method3)_W1] {\WCvM};
			\addlegendentry{W2}
		\end{axis}
	\end{tikzpicture}
	\begin{tikzpicture}
		\tikzset{every node/.style={}}
		\begin{axis}
			[width=6cm,
			ymode=log,
			xmax=0.95,xmin=0.05,
			xlabel={$\rho$},
			ylabel={$W_2$},
			ylabel near ticks,
			legend pos=south east,
			legend style={legend cell align=left,draw=none,fill=white,fill opacity=0.8,text opacity=1,},
			]
			\addplot[color=black, mark=square*] table [x=rho, y=MLE_W2] {\WCvM};
			\addlegendentry{MLE}
			\addplot[color=black, dashed, mark=*] table [x=rho, y=W1(method2)_W2] {\WCvM};
			\addlegendentry{W1}
			\addplot[color=black, dotted, mark=triangle*] table [x=rho, y=W2(method3)_W2] {\WCvM};
			\addlegendentry{W2}
		\end{axis}
	\end{tikzpicture}
	\caption{Comparison of estimators for the von Mises distribution under samples from the wrapped Cauchy distribution.}
	\label{fig_WCvM}
\end{figure}

\section{Discussion}
In this study, we examined the performance of the Wasserstein projection estimators for circular distributions.
Numerical results demonstrate that the Wasserstein projection estimators are close to efficient for the von Mises and wrapped Cauchy distributions, which can be viewed as counterparts of the Gaussian distribution on circles.
This result is analogous to the fact that the $L^2$ Wasserstein projection estimator is asymptotically efficient for the one-dimensional Gaussian distribution \cite{AM2022}.
Numerical results also imply that the Wasserstein projection estimators are robust to noise contamination.
It is an interesting future work to establish theoretical supports of these findings from the viewpoint of robust estimation for circular distributions \cite{Ko,Agostinelli,Kato}.
In addition, experiments under model misspecification showed that the performance of the Wasserstein projection estimators depend on the tail behavior of the circular distributions.

Recently, a unified framework of Wasserstein information geometry based on the Wasserstein information matrix has been developed \cite{WIM,Ay,Nishimori,Fukushi}.
The Wasserstein information matrix is a Riemannian metric that gives a local quadratic approximation of the $L^2$ Wasserstein distance \citep{Ay}. 
Thus, the $L^2$ Wasserstein projection estimator in this study is expected to have some relation to this framework, which is an important future problem.
Note that the Wasserstein estimator proposed by \cite{WIM} is defined as a zero point of the Wasserstein score function, which does not coincide with the Wasserstein projection estimator in general.
It is also of interest to investigate the Wasserstein natural gradient method \citep{Chen} for circular distributions.

\section*{Acknowledgements}
We thank the referees for helpful comments.
Takeru Matsuda was supported by JSPS KAKENHI Grant Numbers 21H05205, 22K17865 and 24K02951 and JST Grant Numbers JPMJMS2024 and JPMJAP25B1.

\appendix

\section{Kullback--Leibler divergence between von Mises and wrapped Cauchy distributions}
This appendix presents analytical derivations of the Kullback--Leibler (KL) divergences between the von Mises distribution $P$ and the wrapped Cauchy distribution $Q$ defined by
\begin{align*}
	p(\theta) &= \frac{1}{2\pi I_0(\kappa)} \exp(\kappa \cos(\theta - \mu_P)), \\
	q(\theta) &= \frac{1 - \rho^2}{2\pi (1 + \rho^2 - 2\rho \cos(\theta - \mu_Q))},
\end{align*}
where $\mu_P, \mu_Q \in [-\pi, \pi)$ are the mean directions, $\kappa > 0$ and $\rho \in [0, 1)$ are the concentration parameters, and $I_n(\kappa)$ is the modified Bessel function of the first kind of order $n$.

Using standard integrations, the trigonometric moments of $P$ and $Q$ are calculated as
\begin{align}
	{\rm E}_P[\cos(n(\theta - \mu_P))] &= \frac{I_n(\kappa)}{I_0(\kappa)}, \quad {\rm E}_P[\sin(n(\theta - \mu_P))] = 0, \label{mom_vM} \\
	{\rm E}_Q[\cos(n(\theta - \mu_Q))] &= \rho^n, \quad {\rm E}_Q[\sin(n(\theta - \mu_Q))] = 0, \label{mom_WC}
\end{align}
for $n \ge 1$.

From the Taylor expansion of the complex logarithm, we have
\begin{equation*}
	\log(1 - \rho e^{ix}) = -\sum_{n=1}^\infty \frac{\rho^n}{n} e^{inx}.
\end{equation*}
Adding this to its complex conjugate yields
\begin{equation*}
	\log(1 - \rho e^{ix}) + \log(1 - \rho e^{-ix}) = -2 \sum_{n=1}^\infty \frac{\rho^n}{n} \cos(nx).
\end{equation*}
Since the left-hand side can be combined as
\begin{equation*}
	\log((1 - \rho e^{ix})(1 - \rho e^{-ix})) = \log(1 + \rho^2 - 2\rho\cos x),
\end{equation*}
we obtain
\begin{equation}\label{denom}
	\log(1 + \rho^2 - 2\rho\cos x) = -2 \sum_{n=1}^\infty \frac{\rho^n}{n} \cos(nx).
\end{equation}

\subsection{Formula of $D_{\text{KL}}(P \parallel Q)$}
First, 
\begin{equation*}
	{\rm E}_P[\log p(\theta)] = -\log(2\pi) - \log I_0(\kappa) + \kappa \frac{I_1(\kappa)}{I_0(\kappa)}.
\end{equation*}
Next, by using \eqref{denom},
\begin{align*}
	{\rm E}_P[\log q(\theta)] &= \log(1 - \rho^2) - \log(2\pi) - {\rm E}_P[\log (1 + \rho^2 - 2\rho\cos (\theta-\mu_Q))] \\
	&= \log(1 - \rho^2) - \log(2\pi) + 2 \sum_{n=1}^\infty \frac{\rho^n}{n} {\rm E}_P [\cos(n(\theta - \mu_Q))].
\end{align*}
From \eqref{mom_vM} and \eqref{mom_WC},
\begin{align*}
	&{\rm E}_P [\cos(n(\theta - \mu_Q))] \\
	=& {\rm E}_P [\cos(n(\theta - \mu_P))] \cos(n(\mu_P-\mu_Q)) - {\rm E}_P [\sin(n(\theta - \mu_P))] \sin(n(\mu_P-\mu_Q))  \\
	=& \frac{I_n(\kappa)}{I_0(\kappa)} \cos(n(\mu_P - \mu_Q)).
\end{align*}
Therefore,
\begin{align*}
	D_{\text{KL}}(P \parallel Q) &= {\rm E}_P[\log p(\theta)] - {\rm E}_P[\log q(\theta)] \\
	&= \kappa \frac{I_1(\kappa)}{I_0(\kappa)} - \log I_0(\kappa) - \log(1 - \rho^2) - 2 \sum_{n=1}^\infty \frac{\rho^n I_n(\kappa)}{n I_0(\kappa)} \cos(n(\mu_P - \mu_Q)).
\end{align*}

\subsection{Formula of $D_{\text{KL}}(Q \parallel P)$}
First, by using \eqref{mom_WC} and \eqref{denom},
\begin{align*}
	{\rm E}_Q[\log q(\theta)] &= \log(1 - \rho^2) - \log(2\pi) - {\rm E}_Q[\log(1 + \rho^2 - 2\rho\cos(\theta - \mu_Q))] \\
	&= \log(1 - \rho^2) - \log(2\pi) + 2 \sum_{n=1}^\infty \frac{\rho^{n}}{n} 
	{\rm E}_Q [\cos(\theta - \mu_Q)] \\
	&= \log(1 - \rho^2) - \log(2\pi) + 2 \sum_{n=1}^\infty \frac{\rho^{2n}}{n}  \\
	&= \log(1 - \rho^2) - \log(2\pi) - 2 \log (1-\rho^2)  \\
	& = -\log(1 - \rho^2) - \log(2\pi).
\end{align*}
Next, by using \eqref{mom_WC},
\begin{align*}
	&{\rm E}_Q[\log p(\theta)] \\
	=& -\log(2\pi) - \log I_0(\kappa) + \kappa {\rm E}_Q[\cos(\theta - \mu_P)] \\
	=& -\log(2\pi) - \log I_0(\kappa) + \kappa ({\rm E}_Q[\cos(\theta - \mu_Q)] \cos(\mu_Q-\mu_P)-{\rm E}_Q[\sin(\theta - \mu_Q)] \sin(\mu_Q-\mu_P)) \\
	=& -\log(2\pi) - \log I_0(\kappa) + \kappa \rho \cos(\mu_P-\mu_Q).
\end{align*}
Therefore,
\begin{align*}
	D_{\text{KL}}(Q \parallel P) &= {\rm E}_Q[\log q(\theta)] - {\rm E}_Q[\log p(\theta)] \\
	&= \log I_0(\kappa) - \log(1 - \rho^2) - \kappa \rho \cos(\mu_P - \mu_Q).
\end{align*}


\begin{thebibliography}{99}
\bibitem[Abe and Pewsey(2011)]{skewed} 
Abe, T. \& Pewsey, A. (2011).
Sine-skewed circular distributions.
\textit{Statistical Papers}, \textbf{52}, 683--707.

\bibitem[Agostinelli(2007)]{Agostinelli} 
Agostinelli, C. (2007).
Robust estimation for circular data.
\textit{Computational Statistics \& Data Analysis}, \textbf{51}, 5867.

\bibitem[Amari and Matsuda(2022)]{AM2022} 
Amari, S. \& Matsuda, T. (2022).
Wasserstein statistics in one-dimensional location scale models
\textit{Annals of the Institute of Statistical Mathematics}, \textbf{74}, 33--47.

\bibitem[Amari and Matsuda(2024)]{affine} 
Amari, S. \& Matsuda, T. (2024).
Information geometry of {W}asserstein statistics on shapes and affine deformations.
\textit{Information Geometry}, \textbf{7}, 285--309.

\bibitem[Ay(2024)]{Ay} 
Ay, N. (2024).
Information geometry of the {O}tto metric.
\textit{Information Geometry}, accepted.

\bibitem[Bassetti et al.(2006)]{BBR2006} 
Bassetti, F., Bodini, A. \& Regazzini, E. (2006).
On minimum Kantorovich distance estimators. 
\textit{Statistics \& Probability Letters}, \textbf{76}, 1298--1302.

\bibitem[Bernton et al.(2019)]{BJGR2019} 
Bernton, E., Jacob, P. E., Gerber, M. \& Robert, C. P. (2019).
On parameter estimation with the Wasserstein distance. 
\textit{Information and Inference: A Journal of the IMA}, \textbf{8}, 657--676.

\bibitem[Chen and Li(2020)]{Chen} 
Chen, Y. \& Li, W. (2020).
Optimal transport natural gradient for statistical manifolds with continuous sample space.
\textit{Information Geometry}, \textbf{3}, 1--32.

\bibitem[Delon et al.(2010)]{Delon} 
Delon, J., Salomon, J. \& Sobolevski, A. (2010).
Fast transport optimization for Monge costs on the circle. 
\textit{SIAM Journal on Applied Mathematics}, \textbf{70}, 2239--2258.

\bibitem[Fournier and Guillin(2015)]{Fournier} 
Fournier, N. \& Guillin, A. (2015).
On the rate of convergence in Wasserstein distance of the empirical measure. 
\textit{Probability Theory and Related Fields}, \textbf{162}, 707--738.

\bibitem[Fronger et al.(2015)]{FZMAP2015} 
Fronger, C., Zhang, C., Mobahi, H., Araya-Polo, M. \& Poggio, T. (2015).
Learning with a Wasserstein loss. 
Advances in Neural Information Processing Systems 28 (NIPS 2015).

\bibitem[Fukushi et al.(2025)]{Fukushi} 
Fukushi, N., Nakanishi-Ohno, Y. \& Matsuda, T. (2025).
Flatness of location-scale-shape models under the Wasserstein metric.
arXiv:2511.09959.

\bibitem[{Huber(2009)}]{Huber}
{Huber, P. J.} \& Ronchetti, E. M. (2009).
\textit{Robust Statistics}.
Wiley.

\bibitem[Hundrieser et al.(2022)]{Hundrieser} 
Hundrieser, S., Klatt, M. \& Munk, A. (2022).
The Statistics of~Circular Optimal Transport.
In \textit{Directional Statistics for Innovative Applications: A Bicentennial Tribute to Florence Nightingale}, Springer, 57--82.

\bibitem[Kato and Eguchi(2016)]{Kato} 
Kato, S. \& Eguchi, S. (2016).
Robust estimation of location and concentration parameters for the von Mises–Fisher distribution.
\textit{Statistical Papers}, \textbf{57}, 205--234.

\bibitem[Kent and Tyler(1988)]{Kent} 
Kent, J. T. \& Tyler, D. E. (1988).
Maximum likelihood estimation for the wrapped Cauchy distribution.
\textit{Journal of Applied Statistics}, \textbf{15}, 247--254.

\bibitem[Ko and Guttorp(1988)]{Ko} 
Ko, D. \& Guttorp, P. (1988).
Robustness of estimators for directional data.
\textit{Annals of Statistics}, \textbf{16}, 609--618.

\bibitem[Li et al.(2018)]{Li} 
Li, W., Ryu, E. K., Osher, S., Yin, W. \& Gangbo, W. (2018).
A parallel method for earth movers distance.
\textit{Journal of Scientific Computing}, \textbf{75}, 182--197.

\bibitem[Li and Zhao(2023)]{WIM} 
Li, W. \& Zhao, J. (2023).
Wasserstein information matrix.
\textit{Information Geometry}, \textbf{6}, 203--255.

\bibitem[{Mardia and Jupp(2008)}]{Mardia}
{Mardia, K. V.} \& {Jupp, P. E.} (2008).
\textit{Directional Statistics}.
New York: Wiley.

\bibitem[McCullagh(1996)]{McCullagh} 
McCullagh, P. (1996).
Möbius transformation and Cauchy parameter estimation. 
\textit{Annals of Statistics}, \textbf{24}, 787--808.

\bibitem[Montavon et al.(2015)]{MMC2015} 
Montavon, G., M\"{u}ller, K. R. \& Cuturi, M. (2015). 
Wasserstein training for Boltzmann machine. 
Advances in Neural Information Processing Systems 29 (NIPS 2016).

\bibitem[Nishomori and Matsuda(2025)]{Nishimori} 
Nishimori, H. \& Matsuda, T. (2025).
On the attainment of the Wasserstein--Cramer--Rao lower bound.
\textit{Information Geometry}, accepted.

\bibitem[Okamura and Otobe(2022)]{Okamura} 
Okamura, K. \& Otobe, Y. (2022).
Characterizations of the maximum likelihood estimator of the Cauchy distribution. 
\textit{Lobachevskii Journal of Mathematics}, \textbf{43}, 25763--2590.

\bibitem[Peyr\'{e} and Cuturi(2019)]{PC2019} 
Peyr\'{e}, G. \& Cuturi, M. (2019).
Computational optimal transport: With Applications to Data Science.
\textit{Foundations and Trends{\textregistered} in Machine Learning}, \textbf{11}, 355--607.

\bibitem[Rabin et al.(2011)]{Rabin} 
Rabin, J., Delon, J. \& Gousseau, Y. (2011).
Transportation distances on the circle.
\textit{Journal of Mathematical Imaging and Vision}, \textbf{41}, 147--167.

\bibitem[van der Vaart(1998)]{vV} 
van der Vaart, A. W. (1998). 
\textit{Asymptotic Statistics}. 
Cambridge University Press.

\bibitem[Villani(2009)]{Villani2009} 
Villani, C. (2009). 
\textit{Optimal Transport: Old and New}. 
Springer.
\end{thebibliography}
\end{document}